\documentclass[12pt]{amsart}
\usepackage{amscd, amsfonts, amssymb, amsthm}
\usepackage[all]{xypic}
\usepackage{graphics, lscape}
\usepackage{rotating}
\renewenvironment{enumerate}
  {\begin{list}
  {(\thesubsection.\theenumi)}
  {\usecounter{enumi}\setlength{\topsep}{8pt plus 2pt}
     \setlength{\parsep}{5.16663pt plus 1pt}
     \setlength{\itemsep}{5.16663pt plus 1pt}
     \setlength{\labelwidth}{2.5em}
     \setlength{\labelsep}{0.5em}
     \setlength{\itemindent}{0pt}
   }}
  {\end{list}}

\newcommand\CB{{\mathcal B}}

\newcommand\CN{{\mathcal N}} 
\newcommand\CO{{\mathcal O}}

\newcommand\fb{{\mathfrak b}}
\newcommand\fg{{\mathfrak g}}

\newcommand\fu{{\mathfrak u}}

\newcommand\ft{{\mathfrak t}}

\newcommand\BBQ{{\mathbb Q}}
\newcommand\BBR{{\mathbb R}}

\newcommand\Coinv{\operatorname{Coinv}}
\newcommand\Ext{{\operatorname{Ext}}}
\newcommand\End{{\operatorname{End}}}

\newcommand\Lie{{\operatorname{Lie}}}

\newcommand\res{\operatorname{res}}

\newcommand\inverse{^{-1}}
\renewcommand\th{{^{\text{th}}}}

\newcommand\bi{{\operatorname{bi}}}
\newcommand\CNt{{\widetilde{\mathcal N}}}
\newcommand\fgt{{\widetilde{\fg}}}
\newcommand\fgrs{\fg_{\rs}}
\newcommand\fgrst{{\widetilde {\fg}_{\rs}}}
\newcommand\id{{id}}
\newcommand\muhat{{\widehat \mu}}
\newcommand\pd{{\operatorname{pd}}}
\newcommand\reg{{\operatorname{reg}}}
\newcommand\rs{{\operatorname{rs}}}
\newcommand\too{\longrightarrow}
\newcommand\wtilde{{\widetilde{w}}}
\newcommand\Zhat{{\widehat Z}}
\newcommand\Zwbar{{\overline{Z_w}}} 

\numberwithin{equation}{section}

\theoremstyle{plain}
\newtheorem{lemma}[equation]{Lemma}
\newtheorem{theorem}[equation]{Theorem}
\newtheorem{corollary}[equation]{Corollary}
\newtheorem{proposition}[equation]{Proposition}

\thanks{The authors would like to thank their charming wives for their
  unwavering support during the preparation of this paper}

\subjclass[2000]{Primary 20G05; Secondary 20F55}

\begin{document}

\title[Homology of the Steinberg Variety] {Homology of the Steinberg
  variety\\ and\\ Weyl group coinvariants}

\author[J.M. Douglass]{J. Matthew Douglass} \address{Department of
  Mathematics\\ University of North Texas\\ Denton TX, USA 76203}
\email{douglass@unt.edu} \urladdr{http://hilbert.math.unt.edu}

\author[G. R\"ohrle]{Gerhard R\"ohrle} \address
{Fakult\"at f\"ur Mathematik, Ruhr-Universit\"at Bochum, D-44780
  Bochum, Germany} \email{gerhard.roehrle@rub.de}
\urladdr{http://www.ruhr-uni-bochum.de/ffm/Lehrstuehle/Lehrstuhl-VI/rubroehrle.html}


\maketitle
\allowdisplaybreaks

\begin{abstract}
  Let $G$ be a complex, connected, reductive algebraic group with Weyl
  group $W$ and Steinberg variety $Z$. We show that the graded
  Borel-Moore homology of $Z$ is isomorphic to the smash product of
  the coinvariant algebra of $W$ and the group algebra of $W$.
\end{abstract}


\section{Introduction}

Suppose $G$ is a complex, reductive algebraic group, $\CB$ is the
variety of Borel subgroups of $G$. Let $\fg$ be the Lie algebra of $G$ and
$\CN$ the cone of nilpotent elements in $\fg$. Let $T^*\CB$ denote the
cotangent bundle of $\CB$. Then there is a \emph{moment map},
$\mu_0\colon T^*\CB\to \CN$. The \emph{Steinberg variety} of $G$ is
the fibered product $T^*\CB\times_\CN T^*\CB$ which we will identify
with the closed subvariety
\[
Z=\{\, (x, B', B'')\in \CN\times \CB\times \CB\mid x\in \Lie(B')\cap
\Lie(B'')\,\}
\]
of $\CN\times \CB\times \CB$.  Set $n= \dim \CB$. Then $Z$ is a
$2n$-dimensional, complex algebraic variety.

If $V= \oplus_{i\geq 0} V_i$ is a graded vector space, we will
frequently denote $V$ by $V_\bullet$. Similarly, if $X$ is a
topological space, then $H_i(X)$ denotes the $i\th$ rational
Borel-Moore homology of $X$ and $H_\bullet(X)=\oplus_{i\geq 0} H_i(X)$
denotes the total Borel-Moore homology of $X$.

Fix a maximal torus, $T$, of $G$, with Lie algebra $\ft$, and let
$W=N_G(T)/T$ be the Weyl group of $(G,T)$. In
\cite{kazhdanlusztig:topological} Kazhdan and Lusztig defined an
action of $W\times W$ on $H_\bullet(Z)$ and they showed that the
representation of $W\times W$ on the top-dimensional homology of $Z$,
$H_{4n}(Z)$, is equivalent to the two-sided regular representation of
$W$. Tanisaki \cite{tanisaki:twisted} and, more recently, Chriss and
Ginzburg \cite{chrissginzburg:representation} have strengthened the
connection between $H_\bullet(Z)$ and $W$ by defining a $\BBQ$-algebra
structure on $H_\bullet(Z)$ so that $H_i(Z) * H_j(Z) \subseteq
H_{i+j-4n}(Z)$. Chriss and Ginzburg
\cite[\S3.4]{chrissginzburg:representation} have also given an
elementary construction of an isomorphism between $H_{4n}(Z)$ and the
group algebra $\BBQ W$.

Let $Z_1$ denote the ``diagonal'' in $Z$:
\[
Z_1=\{\, (x, B', B')\in \CN\times \CB\times \CB\mid x\in \Lie(B')\,\}.
\]
In this paper we extend the results of Chriss and Ginzburg \cite[\S3.4]{chrissginzburg:representation} and show in Theorem \ref{mult} that for any $i$, the convolution product defines an isomorphism $H_i(Z_1) \otimes H_{4n}(Z) \xrightarrow{\sim} H_i(Z)$. It then follows easily that with the convolution product, $H_\bullet(Z)$ is isomorphic to the smash product of the coinvariant algebra of $W$ and the group algebra of $W$.

Precisely, for $0\leq i\leq n$ let $\Coinv_{2i}(W)$ denote the degree $i$ subspace of the rational coinvariant algebra of $W$, so $\Coinv_{2i}(W)$ may be identified with the space of degree $i$, $W$-harmonic polynomials on $\ft$. If $j$ is odd, define $\Coinv_{j}(W)=0$.  Recall that the smash product, $\Coinv(W) \# \BBQ W$, is the $\BBQ$-algebra whose underlying vector space is $\Coinv(W) \otimes_{\BBQ} \BBQ W$ with multiplication satisfying $(f_1\otimes \phi_1) \cdot (f_2\otimes \phi_2)= f_1 \phi_1(f_2) \otimes \phi_1 \phi_2$ where $f_1$ and $f_2$ are in $\Coinv(W)$, $\phi_1$ and $\phi_2$ are in $\BBQ W$, and $\BBQ W$ acts on $\Coinv(W)$ in the usual way. The algebra $\Coinv(W) \# \BBQ W$ is graded by $(\Coinv(W)\# \BBQ W)_i= \Coinv_i(W) \# \BBQ W$ and we will denote this graded algebra by $\Coinv_\bullet(W) \# \BBQ W$. In Theorem \ref{smashiso} we construct an explicit isomorphism of graded algebras $H_{4n-\bullet}(Z) \cong \Coinv_\bullet (W) \# \BBQ W$.

This paper was motivated by the observation, pointed out to the first author by Catharina Stroppel, that the argument in \cite[8.1.5]{chrissginzburg:representation} can be used to show that $H_\bullet(Z)$ is isomorphic to the smash product of $\BBQ W$ and $\Coinv_\bullet(W)$. The details of such an argument have been carried out in a recent preprint of Namhee Kwon \cite{kwon:borel}. This argument relies on some deep and technical results: the localization theorem in $K$-theory proved by Thomason \cite{thomason:formule}, the bivariant Riemann-Roch Theorem \cite[5.11.11]{chrissginzburg:representation}, and the Kazhdan-Lusztig isomorphism between the equivariant $K$-theory of $Z$ and the extended, affine, Hecke algebra \cite{kazhdanlusztig:langlands}. In contrast, and also in the spirit of Kazhdan and Lusztig's original analysis of $H_{4n}(Z)$, and the analysis of $H_{4n}(Z)$ in \cite[3.4]{chrissginzburg:representation}, our argument uses more elementary notions and is accessible to readers who are not experts in equivariant $K$-theory and to readers who are not experts in the representation theory of reductive, algebraic groups.

Another approach to the Borel-Moore homology of the Steinberg variety
uses intersection homology. Let $\mu\colon Z\to \CN$ be projection on
the first factor.  Then, as in
\cite[\S8.6]{chrissginzburg:representation}, $H_\bullet(Z) \cong
\operatorname{Ext} _{D(\CN)} ^{4n-\bullet}\left(R\mu_* \BBQ_\CN,
  R\mu_* \BBQ_\CN \right)$. The Decomposition Theorem of Beilinson,
Bernstein, and Deligne can be used to decompose $R\mu_*\BBQ_\CN$ into
a direct sum of simple perverse sheaves $R\mu_*\BBQ_\CN \cong
\oplus_{x, \phi} \operatorname{IC}_{x,\phi} ^{n_{x, \phi}}$ where $x$
runs over a set of orbit representatives in $\CN$, for each $x$,
$\phi$ runs over a set of irreducible representations of the component
group of $Z_G(x)$, and $\operatorname{IC}_{x,\phi}$ denotes an
intersection complex (see \cite{borhomacpherson:weyl} or
\cite[\S4,5]{shoji:geometry}). Chriss and Ginzburg have used this
construction to describe an isomorphism $H_{4n}(Z) \cong \BBQ W$ and
to in addition give a description of the projective, indecomposable
$H_\bullet(Z)$-modules.

It follows from Theorem \ref{mult} that $H_i(Z)\cong \Coinv_{4n-i}(W)
\otimes H_{4n}(Z)$ and so
\begin{equation}
  \label{eq:ic}
  \Coinv_i(W) \otimes H_{4n}(Z) \cong \operatorname{Ext} _{D(\CN)}
  ^{4n-i}\left(R\mu_* \BBQ_\CN, R\mu_* \BBQ_\CN \right) \cong
  \bigoplus_{x,\phi} \bigoplus_{y, \psi}  \Ext^{4n-i}_{D(\CN)} \left(
    \operatorname{IC}_{x,\phi} ^{n_{x, \phi}},
    \operatorname{IC}_{y,\psi} ^{n_{y, \psi}} \right). 
\end{equation}
In the special case when $i=0$ we have that
\[
\Coinv_0(W) \otimes H_{4n}(Z) \cong \operatorname{End} _{D(\CN)}
\left(R\mu_* \BBQ_\CN\right) \cong \bigoplus_{x,\phi} \End_{D(\CN)}
\left( \operatorname{IC}_{x,\phi} ^{n_{x, \phi}}\right).
\]
The image of the one-dimensional vector space $\Coinv_0(W)$ in
$\operatorname{End} _{D(\CN)} \left(R\mu_* \BBQ_\CN\right)$ is the
line through the identity endomorphism and $\BBQ W\cong H_{4n}(Z)
\cong \oplus_{x,\phi} \operatorname{End} _{D(\CN)}\left(
  \operatorname{IC}_{x,\phi} ^{n_{x,\phi}} \right)$ is the Wedderburn
decomposition of $\BBQ W$ as a direct sum of minimal two-sided ideals.
For $i<4n$ we have not been able to find a nice description of the
image of $\Coinv_i(W)$ in the right-hand side of (\ref{eq:ic}).

The rest of this paper is organized as follows: in \S2 we set up our
notation and state the main results; in \S3 we construct an
isomorphism of graded vector spaces between $\Coinv_\bullet(W) \otimes
\BBQ W$ and $H_{4n-\bullet}(Z)$; and in \S4 we complete the proof that
this isomorphism is in fact an algebra isomorphism when
$\Coinv_\bullet(W) \otimes \BBQ W$ is given the smash product
multiplication. Some very general results about graphs and convolution
that we need for the proofs of the main theorems are proved in an
appendix.

In this paper $\otimes= \otimes_{\BBQ}$, if $X$ is a set, then
$\delta_X$, or just $\delta$, will denote the diagonal embedding of $X$
in $X\times X$, and for $g$ in $G$ and $x$ in $\fg$, $g\cdot x$
denotes the adjoint action of $g$ on $x$.

\section{Preliminaries and Statement of Results}

Fix a Borel subgroup, $B$, of $G$ with $T\subseteq B$ and define $U$
to be the unipotent radical of $B$. We will denote the Lie algebras of
$B$ and $U$ by $\fb$ and $\fu$ respectively.

Our proof that $H_\bullet(Z)$ is isomorphic to $\Coinv_\bullet(W)\#
\BBQ W$ makes use of the specialization construction used by Chriss
and Ginzburg in \cite[\S3.4]{chrissginzburg:representation} to
establish the isomorphism between $H_{4n}(Z)$ and $\BBQ W$. We begin
by reviewing their construction.

The group $G$ acts diagonally on $\CB \times \CB$. Let $\CO_w$ denote
the orbit containing $(B, wBw\inverse)$. Then the rule $w\mapsto
\CO_w$ defines a bijection between $W$ and the set of $G$-orbits in
$\CB \times \CB$.

Let $\pi_Z\colon Z\to \CB \times \CB$ denote the projection on the
second and third factors and for $w$ in $W$ define $Z_w=
\pi_Z\inverse(\CO_w)$.  For $w$ in $W$ we also set $\fu_w= \fu\cap
w\cdot \fu$.  The following facts are well-known (see
\cite{steinberg:desingularization} and \cite[\S1.1]{shoji:geometry}):
\begin{itemize}
\item $Z_w \cong G\times^{B\cap {}^wB} \fu_w$.
\item $\dim Z_w= 2n$.
\item The set $\{\, \Zwbar\mid w\in W\,\}$ is the set of irreducible
  components of $Z$.
\end{itemize}

Define 
\begin{align*}
  \fgt  & = \{\,(x, B')\in \fg\times \CB\mid x\in \Lie(B')\,\},\\
  \CNt  & =\{\,(x, B')\in \CN\times \CB\mid x\in \Lie(B')\,\},\text{
    and}\\ 
  \Zhat & =\{\, (x, B', B'')\in \fg\times \CB \times \CB \mid x\in
  \Lie(B') \cap \Lie(B'')\,\},
\end{align*}
and let $\mu\colon \fgt\to \fg$ denote the projection on the first
factor.  Then $\CNt\cong T^*\CB$, $\mu(\CNt)= \CN$, $Z\cong \CNt\times
_{\CN} \CNt$, and $\Zhat \cong \fgt\times_{\fg} \fgt$.

Let $\hat\pi\colon \Zhat\to \CB \times \CB$ denote the projection on
the second and third factors and for $w$ in $W$ define $\Zhat_w=
\hat\pi\inverse(\CO_w)$. Then it is well-known that $\dim \Zhat_w=
\dim \fg$ and that the closures of the $\Zhat_w$'s for $w$ in $W$ are
the irreducible components of $\Zhat$ (see
\cite[\S1.1]{shoji:geometry}).

Next, for $(x, gBg\inverse)$ in $\fgt$, define $\nu (x,gBg\inverse)$
to be the projection of $g\inverse\cdot x$ in $\ft$. Then $\mu$ and
$\nu$ are two of the maps in Grothendieck's simultaneous resolution:
\[
\xymatrix{ \fgt \ar[r]^{\mu} \ar[d]_{\nu} & \fg \ar[d] \\ \ft \ar[r] &
  \ft/W}
\]

It is easily seen that if $\muhat\colon \Zhat\to \fg$ is the
projection on the first factor, then the square
\[
\xymatrix{ \Zhat \ar[r]^{\muhat} \ar[d] & \fg \ar[d]^{\delta_{\fg}} \\
  \fgt \times \fgt \ar[r]_{\mu\times \mu}& \fg\times \fg}
\]
is cartesian, where the vertical map on the left is given by $(x, B',
B'')\mapsto ((x, B'), (x, B''))$. We will frequently identify $\Zhat$
with the subvariety of $\fgt\times \fgt$ consisting of all pairs
$((x,B'), (x,B''))$ with $x$ in $\Lie(B') \cap \Lie(B'')$.

For $w$ in $W$, let $\Gamma_{w\inverse}= \{\, (h, w\inverse \cdot
h)\mid h\in \ft\,\} \subseteq \ft \times \ft$ denote the graph of the
action of $w\inverse$ on $\ft$ and define
\[
\Lambda_w= \Zhat \cap (\nu\times \nu)\inverse \left(
  \Gamma_{w\inverse} \right) =\{\, (x, B', B'')\in \Zhat \mid
\nu(x,B'')= w\inverse \nu(x,B')\,\}.
\]
In the special case when $w$ is the identity element in $W$, we will
denote $\Lambda_w$ by $\Lambda_1$.

The spaces we have defined so far fit into a commutative diagram with
cartesian squares:
\begin{equation}
  \label{eq:nu}
  \xymatrix{
    \Lambda_w \ar[r] \ar[d] & \Zhat \ar[r]^{\muhat} \ar[d] & \fg
    \ar[d]^{\delta_{\fg}} \\   
    (\nu\times \nu)\inverse \left(\Gamma_{w\inverse} \right) \ar[r]
    \ar[d]  & \fgt \times \fgt \ar[r]_{\mu\times \mu} \ar[d]_{\nu\times
      \nu} & \fg \times \fg \\ 
    \Gamma_{w\inverse} \ar[r] & \ft\times \ft&}  
\end{equation}
Let $\nu_w\colon \Lambda_w\to \Gamma_{w\inverse}$ denote the
composition of the leftmost vertical maps in (\ref{eq:nu}), so $\nu_w$
is the restriction of $\nu\times \nu$ to $\Lambda_w$.

For the specialization construction, we consider subsets of $\Zhat$ of
the form $\nu_w\inverse(S')$ for $S'\subseteq \Gamma_{w\inverse}$.
Thus, for $h$ in $\ft$ we define $\Lambda_w^h= \nu_w\inverse(h,
w\inverse h)$. Notice in particular that $\Lambda_w^0= Z$. More
generally, for a subset $S$ of $\ft$ we define $\Lambda_w^{S}=\coprod
_{h\in S} \Lambda_w^h$. Then, $\Lambda_w^{S}= \nu_w\inverse(S')$,
where $S'$ is the graph of $w\inverse$ restricted to $S$.

Let $\ft_\reg$ denote the set of regular elements in $\ft$.

Fix a one-dimensional subspace, $\ell$, of $\ft$ so that $\ell\cap
\ft_\reg= \ell\setminus \{0\}$ and set $\ell^*= \ell\setminus \{0\}$.
Then $\Lambda_w^{\ell}= \Lambda_w^{\ell^*} \coprod \Lambda_w^0
=\Lambda_w^{\ell^*} \coprod Z$. We will see in Corollary \ref{fib}
that the restriction of $\nu_w$ to $\Lambda_w^{\ell^*}$ is a locally
trivial fibration with fibre $G/T$. Thus, using a construction due to
Fulton and MacPherson (\cite[\S3.4]{fultonmacpherson:categorical},
\cite[\S2.6.30] {chrissginzburg:representation}), there is a
specialization map
\[ 
\lim\colon H_{\bullet+2}( \Lambda_w^{\ell^*}) \too H_{\bullet}(Z).
\]

Since $\Lambda_w^{\ell^*}$ is an irreducible, $(2n+1)$-dimensional
variety, if $[\Lambda_w^{\ell^*}]$ denotes the fundamental class of
$\Lambda_w^{\ell^*}$, then $H_{4n+2}( \Lambda_w^{ \ell^*})$ is
one-dimensional with basis $\{ [\Lambda_w^{\ell^*}] \}$. Define
$\lambda_w= \lim ([\Lambda_w^{\ell^*}])$ in $H_{4n}(Z)$. Chriss and
Ginzburg \cite[\S3.4]{chrissginzburg:representation} have proved the
following theorem.

\begin{theorem}\label{cg}
  Consider $H_\bullet(Z)$ endowed with the convolution product.
  \begin{enumerate}
  \item[(a)] For $0\leq i,j\leq 4n$, $H_i(Z) \,*\, H_j(Z) \subseteq
    H_{i+j-4n}(Z)$. In particular, $H_{4n}(Z)$ is a subalgebra of
    $H_\bullet(Z)$.
  \item[(b)] The element $\lambda_w$ in $H_{4n}(Z)$ does not depend on
    the choice of $\ell$.
  \item[(c)] The assignment $w\mapsto \lambda_w$ extends to an algebra
    isomorphism $\alpha\colon \BBQ W \xrightarrow{\,\cong\,} H_{4n}(Z)$.
  \end{enumerate}  
\end{theorem}

Now consider
\[
Z_1 =\{\, (x, B', B')\in \CN\times \CB\times \CB\mid x\in
\Lie(B')\,\}.
\]
Then $Z_1$ may be identified with the diagonal in $\CNt\times \CNt$.
It follows that $Z_1$ is closed in $Z$ and isomorphic to $\CNt$.

Since $\CNt \cong T^*\CB$, it follows from the Thom isomorphism in
Borel-Moore homology \cite[\S2.6]{chrissginzburg:representation} that
$H_{i+2n}(Z_1) \cong H_{i}(\CB)$ for all $i$. Since $\CB$ is smooth
and compact, $H_{i}(\CB)\cong H^{2n-i}(\CB)$ by Poincar\'e duality.
Therefore, $H_{4n-i}(Z_1) \cong H^i(\CB)$ for all $i$.

The cohomology of $\CB$ is well-understood: there is an isomorphism of
graded algebras, $H^{\bullet}(\CB)\cong \Coinv_\bullet(W)$.  It
follows that $H_j(Z) = 0$ if $j$ is odd and $H_{4n-2i}(Z_1)\cong
\Coinv_{2i}(W)$ for $0\leq i\leq n$.

In \S3 below we will prove the following theorem.

\begin{theorem}\label{mult}
  Consider the Borel-Moore homology of the variety $Z_1$.
  \begin{enumerate}
  \item[(a)] There is a convolution product on $H_\bullet(Z_1)$. With
    this product, $H_{\bullet}(Z_1)$ is a commutative $\BBQ$-algebra
    and there is an isomorphism of graded $\BBQ$-algebras
    \[
    \beta\colon \Coinv_\bullet(W) \xrightarrow{\ \cong\ }
      H_{4n-\bullet}(Z_1).
    \]
  \item[(b)] If $r\colon Z_1\to Z$ denotes the inclusion, then the
    direct image map in Borel-Moore homology, $r_*\colon
    H_\bullet(Z_1) \too H_\bullet(Z)$, is an injective ring
    homomorphism.
  \item[(c)] If we identify $H_\bullet(Z_1)$ with its image in
    $H_\bullet(Z)$ as in (b), then the linear transformation given by
    the convolution product
    \[
    H_i(Z_1)\otimes H_{4n}(Z) \xrightarrow{\,\ *\ \,}  H_i(Z)
    \]
    is an isomorphism of vector spaces for $0\leq i\leq 4n$.
  \end{enumerate}
\end{theorem}

The algebra $\Coinv_\bullet(W)$ has a natural action of $W$ by algebra
automorphisms, and the isomorphism $\beta$ in Theorem \ref{mult}(a) is
in fact an isomorphism of $W$-algebras. The $W$-algebra structure on
$H_\bullet(Z_1)$ is described in the next theorem, which will be
proved in \S4.

\begin{theorem}\label{conj}
  If $w$ is in $W$ and $H_\bullet(Z_1)$ is identified with its image
  in $H_\bullet(Z)$, then
  \begin{equation*}
    \lambda_w*H_i(Z_1) *\lambda_{w\inverse} = H_i(Z_1).
  \end{equation*}
  Thus, conjugation by $\lambda_w$ defines a $W$-algebra structure on
  $H_\bullet(Z_1)$. With this $W$-algebra structure, the isomorphism
  $\beta\colon \Coinv_\bullet(W) \xrightarrow{\ \cong\ } H_{4n-
    \bullet} (Z_1)$ in Theorem \ref{mult}(a) is an isomorphism of
  $W$-algebras.
\end{theorem}

Recall that $\Coinv(W) \# \BBQ W$ is graded by $(\Coinv(W) \# \BBQ
W)_i= \Coinv_i(W) \otimes \BBQ W$. Then combining Theorem \ref{cg}(c),
Theorem \ref{mult}(c), and Theorem \ref{conj} we get our main result.

\begin{theorem}\label{smashiso}
  The composition
  \[
  \Coinv_\bullet(W) \# \BBQ W \xrightarrow{\,\beta\otimes \alpha\,}
  H_{4n-\bullet}(Z_1) \otimes H_{4n}(Z) \xrightarrow{\ *\ }
  H_{4n-\bullet}(Z)
  \]
  is an isomorphism of graded $\BBQ$-algebras.
\end{theorem}

\section{Factorization of $H_{\bullet}(Z)$}

\subsection*{Proof of Theorem \ref{mult}(a)}
We need to prove that $H_{\bullet}(Z_1)$ is a commutative
$\BBQ$-algebra and that $ \Coinv_\bullet(W) \cong H_{4n-\bullet}(Z_1)$.

Let $\pi\colon \CNt\to \CB$ by $\pi(x, B')=B'$. Then $\pi$ may be
identified with the vector bundle projection $T^*\CB\to \CB$ and so
the induced map in cohomology $\pi^*\colon H^i(\CB)\to H^i(\CNt)$ is
an isomorphism. The projection $\pi$ determines an isomorphism in
Borel-Moore homology that we will also denote by $\pi^*$ (see
\cite[\S2.6.42] {chrissginzburg:representation}). We have
$\pi^*\colon H_i(\CB) \xrightarrow{\ \cong\ }  H_{i+2n}(\CNt)$.

For a smooth $m$-dimensional variety $X$, let $\pd\colon H^i(X) \too
H_{2m-i}(X)$ denote the Poincar\'e duality isomorphism. Then the
composition
\[
H_{2n-i}(\CB) \xrightarrow{\ \pd\inverse\ } H^i(\CB) \xrightarrow{\
  \pi^*\ } H^i(\CNt) \xrightarrow{\ \pd\ } H_{4n-i}(\CNt)
\]
is an isomorphism. It follows from the uniqueness construction in
\cite[\S2.6.26] {chrissginzburg:representation} that
\[
\pd \circ \pi^* \circ \pd\inverse= \pi^* \colon H_{2n-i}(\CB)\too
H_{4n-i}(\CNt)
\]
and so $\pi^* \circ \pd= \pd\circ \pi^*\colon H^i(\CB)\too
H_{4n-i}(\CNt)$.

Recall that $\Coinv_{j}(W) = 0$ if $j$ is odd and $\Coinv_{2i}(W)$ is
the degree $i$ subspace of the coinvariant algebra of $W$. Let
$\bi\colon \Coinv_\bullet(W)\too H^\bullet(\CB)$ be the Borel
isomorphism (see \cite[\S1.5]{borhobrylinskimacpherson:nilpotent} or
\cite{hiller:geometry}). Then with the cup product, $H^\bullet(\CB)$
is a graded algebra and $\bi$ is an isomorphism of graded algebras.

Define $\beta\colon \Coinv_{i}(W) \to H_{4n-i}(Z_1)$ to be the
composition
\[
\Coinv_{i}(W) \xrightarrow{\ \bi\ } H^{i}(\CB) \xrightarrow{\ \pi^*\ }
H^{i}(\CNt) \xrightarrow{\ \pd\ } H_{4n-i}(\CNt) \xrightarrow{\
  \delta_*\ } H_{4n-i}(Z_1)
\]
where $\delta=\delta_{\CNt}$. Then $\beta$ is an isomorphism of graded
vector spaces and
\[
\beta= \delta_*\circ \pd\circ \pi^* \circ \bi= \delta_* \circ \pi^*
\circ \pd \circ \bi.
\]

The algebra structure of $H^\bullet(\CB)$ and $H^\bullet(\CNt)$ is
given by the cup product, and $\pi^*\colon H^\bullet(\CB)\to
H^\bullet(\CNt)$ is an isomorphism of graded algebras.  
Since $\CNt$ is smooth, as in
\cite[\S2.6.15]{chrissginzburg:representation}, there is an
intersection product defined on $H_\bullet(\CNt)$ using Poincar\'e
duality and the cup product on $H^\bullet(\CNt)$. Thus, $\pd\colon
H^\bullet(\CNt)\to H_{4n-\bullet}( \CNt)$ is an algebra isomorphism.
Finally, it is observed in
\cite[\S2.7.10]{chrissginzburg:representation} that $\delta_*\colon
H_\bullet(\CNt)\to H_\bullet(Z_1)$ is a ring homomorphism and hence an
algebra isomorphism. This shows that $\beta$ is an isomorphism of
graded algebras and proves Theorem \ref{mult}(a).

\subsection*{Proof of Theorem \ref{mult}(b)} 
To prove the remaining parts of Theorem \ref{mult}, we need a linear
order on $W$. Suppose $|W|=N$. Fix a linear order on $W$ that extends
the Bruhat order. Say $W=\{w_1, \dots, w_N\}$, where $w_1 = 1$ and
$w_N$ is the longest element in $W$.

For $1\leq j\leq N$, define $Z_j= \coprod_{i=1}^j Z_{w_i}$. Then, for
each $j$, $Z_j$ is closed in $Z$, $Z_{w_j}$ is open in $Z_j$, and
$Z_j= Z_{j-1} \coprod Z_{w_j}$. Notice that $Z_N=Z$ and $Z_1=Z_{w_1}$.

Similarly, define $\Zhat_j= \coprod_{i=1}^j \Zhat_{w_i}$. Then each
$\Zhat_j$ is closed in $\Zhat$, $\Zhat_{w_j}$ is open in $\Zhat_j$,
and $\Zhat_j= \Zhat_{j-1} \coprod \Zhat_{w_j}$.

We need to show that $r_*\colon H_\bullet(Z_1) \too H_\bullet(Z)$ is
an injective ring homomorphism.

Let $\res_j \colon H_i(Z_j) \to H_i(Z_{w_j})$ denote the restriction
map in Borel-Moore homology induced by the open embedding $Z_{w_j}
\subseteq Z_j$ and let $r_{j} \colon H_i(Z_{j-1}) \too H_i(Z_{j})$
denote the direct image map in Borel-Moore homology induced by the
closed embedding $Z_{j-1} \subseteq Z_j$. Then there is a long exact
sequence in homology
\[
\xymatrix@1{ \dotsm \ar[r] & H_i(Z_{j-1}) \ar[r]^{r_{j}} & H_i(Z_j)
  \ar[r]^-{\res_j} & H_i(Z_{w_j}) \ar[r]^-{\partial} &
  H_{i-1}(Z_{j-1}) \ar[r] & \dotsm }
\]
It is shown in \cite[\S6.2]{chrissginzburg:representation} that
$\partial=0$ and so the sequence
\begin{equation}
  \label{eq:ses}
  \xymatrix@1{ 0 \ar[r] & H_i(Z_{j-1}) \ar[r]^-{r_{j}} & H_i(Z_j)
    \ar[r]^-{\res_j} & H_i(Z_{w_j}) \ar[r] & 0 
  } 
\end{equation}
is exact for every $i$ and $j$. Therefore, if $r\colon Z_j\to Z$
denotes the inclusion, then the direct image $r_* \colon H_i(Z_j)\to
H_i(Z)$ is an injection for all $i$. (The fact that $r$ depends on $j$
should not lead to any confusion.)

We will frequently identify $H_i(Z_j)$ with its image in $H_i(Z)$ and
consider $H_i(Z_j)$ as a subset of $H_i(Z)$. Thus, we have a flag of
subspaces $0 \subseteq H_i(Z_1) \subseteq \dotsm \subseteq
H_i(Z_{N-1})\subseteq H_i(Z)$.

In particular, $r_* \colon H_i(Z_1)\to H_i(Z)$ is an injection for all
$i$. It follows from \cite[Lemma 5.2.23]
{chrissginzburg:representation} that $r_*$ is a ring homomorphism.
This proves part (b) of Theorem \ref{mult}.

\subsection*{Proof of Theorem \ref{mult}(c)}
We need to show that the linear transformation given by the
convolution product $H_i(Z_1)\otimes H_{4n}(Z) \to H_i(Z)$ is an
isomorphism of vector spaces for $0\leq i\leq 4n$.

The proof is a consequence of the following lemma.

\begin{lemma}\label{surjj}
  The image of the convolution map $* : H_i(Z_1) \otimes H_{4n}(Z_j)
  \too H_i(Z)$ is precisely $H_i(Z_j)$ for $0\leq i\leq 4n$ and $1\leq
  j\leq N$.
\end{lemma}

Assuming that the lemma has been proved, taking $j=N$, we conclude
that the convolution product in $H_\bullet (Z)$ induces a surjection
$H_i(Z_1) \otimes H_{4n}(Z)\too H_i(Z)$. It is shown in \cite[\S6.2]
{chrissginzburg:representation} that $\dim H_\bullet(Z)= |W|^2$ and so
$\dim H_\bullet(Z_1)\otimes H_{4n}(Z)= |W|^2= \dim H_\bullet(Z)$.
Thus, the convolution product induces an isomorphism $H_i(Z_1) \otimes
H_{4n}(Z) \cong H_i(Z)$.

The rest of this section is devoted to the proof of Lemma \ref{surjj}.

To prove Lemma \ref{surjj} we need to analyze the specialization map,
$\lim\colon H_{\bullet+2}( \Lambda_w^{\ell^*}) \to H_{\bullet}(Z)$,
beginning with the subvarieties $\Lambda_w^{\ell}$ and $\Lambda_w^{
  \ell^*}$ of $\Lambda_w$.

\subsection*{Subvarieties of $\Lambda_w$}
Suppose that $\ell$ is a one-dimensional subspace of $\ft$ with
$\ell^* = \ell\setminus \{0\} =\ell\cap \ft_\reg$.  Recall that
$\fu_w= \fu \cap w\cdot \fu$ for $w$ in $W$.

\begin{lemma}\label{lemsat}
  The variety $\Lambda_w^{\ell} \cap \Zhat_w$ is the $G$-saturation in
  $\Zhat$ of $\{\,(h+n, B, wBw\inverse) \mid h\in \ell,\, n\in
  \fu_w\,\}$.
\end{lemma}

\begin{proof}
  By definition,
  \[
  \Lambda_w^{\ell}= \Lambda_w^{\ell^*} \coprod \Lambda_w^0= \{\, (x,
  B', B'')\in \Zhat \mid \nu(x,B'')= w\inverse \nu(x,B')\in w\inverse
  (\ell)\,\}.
  \]

  Suppose that $h$ is in $\ft_\reg$ and $(x, g_1Bg_1\inverse,
  g_2Bg_2\inverse)$ is in $\Lambda_w^h$. Then $g_1\inverse \cdot x=
  h+n_1$ and $g_2\inverse \cdot x= w\inverse h+n_2$ for some $n_1$ and
  $n_2$ in $\fu$. Since $h$ is regular, there are elements $u_1$ and
  $u_2$ in $U$ so that $u_1\inverse g_1\inverse \cdot h= h$ and
  $u_2\inverse g_2\inverse \cdot h= w\inverse h$. Then $x= g_1u_1
  \cdot h= g_2u_2 w\inverse \cdot h$ and so $g_1u_1= g_2u_2 w\inverse
  t$ for some $t$ in $T$. Therefore, $(x, g_1Bg_1\inverse,
  g_2Bg_2\inverse)= g_1u_1 \cdot (h, B, wBw\inverse)$. Thus,
  $\Lambda_w^h$ is contained in the $G$-orbit of $(h, B,wBw\inverse)$.
  Since $\nu$ is $G$-equivariant, it follows that $\Lambda_w^h$ is
  $G$-stable and so $\Lambda_w^h$ is the full $G$-orbit of $(h,
  B,wBw\inverse)$.  Therefore, $\Lambda_w^{\ell^*}$ is the
  $G$-saturation of $\{\,(h+n, B, wBw\inverse\mid h\in \ell^*,\, n\in
  \fu_w\,\}$ and $\Lambda_w^h\subseteq \Zhat_w$ for $h$ in $\ell^*$.

  We have already observed that $\Lambda_w^0=Z$ and so
  \[
  \Lambda_w^{\ell} \cap \Zhat_w= \left(\Lambda_w^{\ell^*} \cap \Zhat_w
  \right)\coprod \left( \Lambda_w^0 \cap \Zhat_w \right)=
  \Lambda_w^{\ell^*} \coprod Z_w.
  \] 
  It is easy to see that $Z_w$ is the $G$-saturation of $\{\,(n, B,
  wBw\inverse) \mid n\in \fu_w\,\}$ in $Z$. This proves the lemma.
\end{proof}

\begin{corollary}\label{corfb}
  The variety $\Lambda_w^{\ell} \cap \Zhat_w$ is a locally trivial,
  affine space bundle over $\CO_w$ with fibre isomorphic to $\ell+
  \fu_w$, and so there is an isomorphism $\Lambda_w^{\ell} \cap
  \Zhat_w \cong G\times^{{B\cap {}^wB}} \left( \ell+ \fu_w\right)$.
\end{corollary}

\begin{proof}
  It follows from Lemma \ref{lemsat} that the map given by projection
  on the second and third factors is a $G$-equivariant morphism from
  $\Lambda_w^\ell$ onto $\CO_w$ and that the fibre over $(B,
  wBw\inverse)$ is $\{\, (h+n, B, wBw\inverse)\mid h\in \ell,\, n \in
  \fu_w\,\}$. Therefore, $\Lambda_w^\ell \cong G\times^{B\cap {}^wB}
  \left( \ell + \fu_w\right)$.
\end{proof}

Let $\fgrs$ denote the set of regular semisimple elements in $\fg$ and
define $\fgrst=\{\, (x, B')\in \fgt \mid x\in \fgrs\,\}$.  For an
arbitrary subset $S$ of $\ft$, define $\fgt^S=\nu\inverse(S)= \{\, (x,
B')\in \fgt \mid \nu(x,B')\in S\,\}$.

For $w$ in $W$, define $\wtilde\colon G/T\times \ft_\reg\too G/T\times
\ft_\reg$
by $\wtilde(gT, h)= (gwT, w\inverse h)$. The rule $(gT, h)\mapsto
(g\cdot h, gB)$ defines an isomorphism of varieties $f\colon G/T\times
\ft_\reg \xrightarrow{\ \cong\ } \fgrst$ and we will denote the
automorphism $f\circ \wtilde \circ f\inverse$ of $\fgrst$ also by
$\wtilde$. Notice that if $h$ is in $\ft_\reg$ and $g$ is in $G$, then
$\wtilde(g\cdot h, gB)= (g\cdot h, gwBw\inverse g\inverse)$.

\begin{lemma}\label{lemgraph}
  The variety $\Lambda_w^{\ell^*}$ is the graph of $\wtilde|_{
    \fgt^{\ell^*}}\colon \fgt^{\ell^*}\to \fgt^{w\inverse(\ell^*)}$.
\end{lemma}

\begin{proof}
  It follows from Lemma \ref{lemsat} that
  \begin{align*}
    \Lambda_w^{\ell^*} &= \{\, (g \cdot h, gBg\inverse, gwBw\inverse
    g\inverse) \in \fgrs\times \CB\times \CB \mid h\in \ell^*,\, g\in
    G\,\}\\
    &= \{\, ((g\cdot h, gBg\inverse), (g\cdot h, gwBw\inverse
    g\inverse)) \in \fgt\times \fgt \mid h\in \ell^*,\, g\in G\,\}.
  \end{align*}
  The argument in the proof of Lemma \ref{lemsat} shows that
  $\fgt^{\ell^*} = \{\, (g\cdot h, gBg\inverse) \mid h\in \ell^*,\,
  g\in G\,\}$ and by definition $\wtilde(g \cdot h, gB)= (g\cdot h,
  gwBw\inverse g\inverse)$. Therefore, $\Lambda_w^{\ell^*}$ is the
  graph of $\wtilde|_{ \fgt^{\ell^*}}$.
\end{proof}

\begin{corollary}\label{fib}
  The map $\nu_w\colon \Lambda_w^{\ell^*}\to \ell^*$ is a locally
  trivial fibration with fibre isomorphic to $G/T$.
\end{corollary}

\begin{proof}
  This follows immediately from the lemma and the fact that
  $\fgt^{\ell^*} \cong G /T\times \ell^*$.
\end{proof}

\subsection*{The specialization map}
Suppose that $w$ is in $W$ and that $\ell$ is a one-dimensional
subspace of $\ft$ with $\ell^* = \ell\setminus \{0\} =\ell\cap
\ft_\reg$.  As in \cite{fultonmacpherson:categorical} and
\cite[\S2.6.30] {chrissginzburg:representation}, $\lim \colon H_{i+2}(
\Lambda_w^{\ell^*}) \to H_i(Z)$ is the composition of three maps,
defined as follows.

As a vector space over $\BBR$, $\ell$ is two-dimensional. Fix an
$\BBR$-basis of $\ell$, say $\{v_1, v_2\}$. Define $P$ to be the open
half plane $\BBR_{>0} v_1\oplus \BBR v_2$, define $I_{>0}$ to be the
ray $\BBR_{>0} v_1$, and define $I$ to be the closure of $I_{>0}$, so
$I= \BBR_{\geq 0} v_1$.

Since $P$ is an open subset of $\ell^*$, $\Lambda_w^{P}$ is an open
subset of $\Lambda_w^{\ell^*}$ and so there is a restriction map in
Borel-Moore homology $\res\colon H_{i+2}(\Lambda_w^{\ell^*}) \to
H_{i+2} (\Lambda_w^{P})$.

The projection map from $P$ to $I_{>0}$ determines an isomorphism in
Borel-Moore homology $\psi\colon H_{i+2} (\Lambda_w^{P}) \to H_{i+1}
(\Lambda_w^{I_{>0}})$.

Since $I= I_{>0}\coprod \{0\}$, we have $\Lambda_w^I= \Lambda_w^{
  I_{>0}} \coprod \Lambda_w^0= \Lambda_w^{I_{>0}} \coprod Z$, where
$Z$ is closed in $\Lambda_w^I$. The connecting homomorphism of the
long exact sequence in Borel-Moore homology arising from the partition
$\Lambda_w^I= \Lambda_w^{I_{>0}} \coprod Z$ is a map $\partial \colon
H_{i+1} (\Lambda_w^{I_{>0}}) \to H_i(Z)$.

By definition, $\lim= \partial \circ \psi\circ \res$.

Now fix $j$ with $1\leq j\leq N$ and set $w=w_j$.

Consider the intersection $\Lambda_w^I\cap \Zhat_j= (\Lambda_w^{
  I_{>0}} \cap \Zhat_j) \coprod (Z\cap \Zhat_j)$. Then $Z\cap \Zhat_j$
is closed in $\Lambda_w^I\cap \Zhat_j$ and by construction,
$\Lambda_w^{I_{>0}} \subseteq \Zhat_j$ and $Z\cap \Zhat_j=Z_j$. Thus,
$\Lambda_w^I\cap \Zhat_j= \Lambda_w^{ I_{>0}} \coprod Z_j$.  Let
$\partial_j\colon H_{i+1}(\Lambda_w^{ I_{>0}}) \to H_i(Z_j)$ be the
connecting homomorphism of the long exact sequence in Borel-Moore
homology arising from this partition. Because the long exact sequence
in Borel-Moore homology is natural, we have a commutative square:
\[
\xymatrix{ H_{i+1}(\Lambda_w^{I_{>0}}) \ar[r]^-{\partial} & H_i(Z) \\
  H_{i+1}(\Lambda_w^{I_{>0}}) \ar[r]_-{\partial_j} \ar@{=}[u] &
  H_i(Z_j) \ar[u]_{r_*} }
\]
This proves the following lemma.

\begin{lemma}
  Fix $j$ with $1\leq j\leq N$ and set $w=w_j$. Then $\partial \colon
  H_{i+1} (\Lambda_w^{I_{>0}}) \too H_i(Z)$ factors as $r_* \circ
  \partial_j$ where $\partial_j\colon H_{i+1} (\Lambda_w^{I_{>0}}) \too
  H_i(Z_j)$ is the connecting homomorphism of the long exact sequence
  arising from the partition $\Lambda_w^I\cap \Zhat=
  \Lambda_w^{I_{>0}} \coprod Z_j$.
\end{lemma}

It follows from the lemma that $\lim\colon H_{i+2}(\Lambda_w
^{\ell^*}) \too H_i(Z)$ factors as
\begin{equation} 
  \label{eq:fact} 
  H_{i+2}(\Lambda_w
  ^{\ell^*}) \xrightarrow{\ \res\ } H_{i+2} (\Lambda_w^{P})
  \xrightarrow{\ \psi\ }
  H_{i+1} (\Lambda_w^{I_{>0}}) \xrightarrow{\ \partial_j\ } H_i(Z_j)
  \xrightarrow{\ r_*\ } H_i(Z).
\end{equation}
Define $\lim_j \colon H_{i+2}(\Lambda_w ^{\ell^*}) \too H_i(Z_j)$ by
$\lim_j= \partial_j \circ \psi\circ \res$.

\subsection*{Specialization and restriction}
As above, fix $j$ with $1\leq j\leq N$ and a one-dimensional subspace
$\ell$ of $\ft$ with $\ell^* = \ell\setminus \{0\} =\ell\cap
\ft_\reg$.  Set $w = w_j$.

Recall the restriction map $\res_j\colon H_i(Z_j) \to H_i(Z_{w})$
from (\ref{eq:ses}).

\begin{lemma} \label{lem3.6} 
  The composition $\res_j\circ \lim\nolimits_j \colon
  H_{i+2}(\Lambda_w ^{\ell^*}) \too H_i(Z_{w})$ is surjective for
  $0\leq i\leq 4n$.
\end{lemma}

\begin{proof}
  Using (\ref{eq:fact}), $\res_j\circ \lim\nolimits_j$ factors as
  \[
  H_{i+2}(\Lambda_w ^{\ell^*}) \xrightarrow{\ \res\ } H_{i+2}
  (\Lambda_w^{P}) \xrightarrow{\ \psi\ } H_{i+1} (\Lambda_w^{I_{>0}})
  \xrightarrow{\ \partial_j\ } H_i(Z_j) \xrightarrow{\ \res_j\ }
  H_i(Z_w).
  \]
  Lemma \ref{lemalways} below shows that $\res$ is always surjective
  and the map $\psi$ is an isomorphism, so we need to show that the
  composition $\res_j\circ\ \partial_j$ is surjective.

  Consider $\Lambda_w^I\cap \Zhat_w= (\Lambda_w^I\cap \Zhat_j)\cap
  \Zhat_w= \Lambda_w^{I_{>0}} \coprod Z_w$. Then $\Lambda_w^{I_{>0}}$
  is open in $\Lambda_w^I \cap \Zhat_w$ and we have a commutative
  diagram of long exact sequences
  \[
  \xymatrix{\dotsm \ar[r]& H_{i+1}( \Lambda_w^I \cap \Zhat_w) \ar[r]&
    H_{i+1}( \Lambda_w^{I_{>0}}) \ar[r]^{\partial_w}& H_i(Z_w) \ar[r]&
    \dotsm \\
    \dotsm \ar[r]& H_{i+1}( \Lambda_w^I \cap \Zhat_j) \ar[r] \ar[u]&
    H_{i+1}( \Lambda_w^{I_{>0}}) \ar[r]^{\partial_j} \ar@{=}[u] &
    H_i(Z_j) \ar[r] \ar[u]_{\res_j}& \dotsm}
  \] 
  where $\partial_w$ is the connecting homomorphism of the long exact
  sequence arising from the partition $\Lambda_w^I\cap \Zhat_w=
  \Lambda_w^{I_{>0}} \coprod Z_w$. We have seen at the beginning of
  this section that $\res_j$ is surjective and so it is enough to show
  that $\partial_w$ is surjective.

  Recall that $\{v_1, v_2\}$ is an $\BBR$-basis of $\ell$ and $I=
  \BBR_{\geq 0} v_1$. Define
  \begin{align*}
    E_{I} & = G\times^{{B\cap {}^wB}} \left( \BBR_{\geq 0} v_1+
      \fu_w \right),\\
    E_{I_{>0}} & = G\times^{{B\cap {}^wB}} \left( \BBR_{>0} v_1+
      \fu_w \right),\ \text{and}\\
    E_{0} & = G\times^{{B\cap {}^wB}} \fu_w.
  \end{align*}
  It follows from Corollary \ref{corfb} that $E_I \cong \Lambda_w^I$,
  $E_{I_{>0}} \cong \Lambda_w^{I_{>0}}$, and $E_0 \cong Z_w$, so the
  long exact sequence arising from the partition $\Lambda_w^I\cap
  \Zhat_w =\Lambda_w^{I_{>0}} \coprod Z_w$ may be identified with the
  long exact sequence arising from the partition $E_I= E_{I_{>0}}
  \coprod E_0$:
  \[
  \xymatrix@1{\dotsm \ar[r]& H_{i+1}(E_I) \ar[r]& H_{i+1}(E_{I_{>0}})
    \ar[r]^-{\partial_E}& H_i(E_0) \ar[r]& \dotsm }
  \]
  Therefore, it is enough to show that $\partial_E$ is surjective. In
  fact, we show that $H_\bullet(E_I) = 0$ and so $\partial_E$ is an
  isomorphism.

  Define $E_{\BBR}= G\times^{{B\cap {}^wB}} \left( \BBR v_1+
    \fu_w\right)$. Then $E_\BBR$ is a smooth, real vector bundle over
  $G/ B\cap {}^wB$ and so $E_\BBR$ is a smooth manifold containing
  $E_I$ as a closed subset. We may apply \cite[2.6.1]
  {chrissginzburg:representation} and conclude that $H_i(E_I) \cong
  H^{4n+1-i}( E_\BBR, E_\BBR \setminus E_I)$.

  Consider the cohomology long exact sequence of the pair $(E_\BBR,
  E_\BBR \setminus E_I)$. Since $E_\BBR$ is a vector bundle over $G/
  B\cap {}^wB$, it is homotopy equivalent to $G/ B\cap {}^wB$.
  Similarly, $E_\BBR \setminus E_I \cong G\times^{{B\cap {}^wB}}
  \left( \BBR_{<0} v_1+ \fu_w\right)$ and so is also homotopy
  equivalent to $G/ B\cap {}^wB$. Therefore, $H^i(E_\BBR) \cong H^i(
  E_\BBR \setminus E_I)$ and it follows that the relative cohomology
  group $H^i( E_\BBR, E_\BBR \setminus E_I)$ is trivial for every $i$.
  Therefore, $H_\bullet(E_I) = 0$, as claimed.

  This completes the proof of the lemma.
\end{proof}

\begin{corollary}\label{lim1}
  The specialization map $\lim\nolimits_1 \colon H_{i+2}(\Lambda_1
  ^{\ell^*}) \too H_i(Z_1)$ is surjective for $0\leq i\leq 4n$.
\end{corollary}

\begin{proof}
  This follows from Lemma \ref{lem3.6}, because $Z_1= Z_{w_1}$ and so
  $\res_1$ is the identity map.
\end{proof}

The next lemma is true for any specialization map.

\begin{lemma}\label{lemalways}
  The restriction map $\res\colon H_{i+2}(\Lambda_w^{\ell^*}) \too
  H_{i+2} (\Lambda_w^{P})$ is surjective for every $w$ in $W$ and
  every $i\geq 0$.
\end{lemma}

\begin{proof}
  There are homeomorphisms $\Lambda_w^{\ell^*} \cong G/T \times
  \ell^*$ and $\Lambda_w^{P} \cong G/T \times P$. By definition, $P$
  is an open subset of $\ell^*$ and so there is a restriction map
  $\res\colon H_2(\ell^*)\to H_2(P)$. This map is a non-zero linear
  transformation between one-dimensional $\BBQ$-vector spaces so it is
  an isomorphism.

  Using the K\"unneth formula we get a commutative square where the
  horizontal maps are isomorphisms and the right-hand vertical map is
  surjective:
  \[
  \xymatrix{ H_{i+2}(\Lambda_w^{\ell^*}) \ar[rr]^-{\cong}
    \ar[d]_{\res} && H_i(G/T)\otimes H_2(\ell^*) + H_{i+1}(G/T)\otimes
    H_1(\ell^*)
    \ar[d]^{\id \otimes \res + 0} \\
    H_{i+2} (\Lambda_w^{P}) \ar[rr]_-{\cong} && H_i(G/T) \otimes
    H_2(P) }
  \]
  It follows that $\res\colon H_{i+2}(\Lambda_w^{\ell^*}) \to H_{i+2}
  (\Lambda_w^{P})$ is surjective.
\end{proof}

\subsection*{Proof of Lemma \ref{surjj}}
Fix $i$ with $0\leq i\leq 4n$. We show that the image of the
convolution map $* : H_i(Z_1) \otimes H_{4n}(Z_j) \too H_i(Z)$ is
precisely $H_i(Z_j)$ for $1\leq j\leq N$ using induction on $j$.

For $j=1$, $H_{4n}(Z_1)$ is one-dimensional with basis
$\{\lambda_1\}$. It follows from Theorem \ref{cg}(c) that $\lambda_1$
is the identity in $H_\bullet (Z)$ and so clearly the image of the
convolution map $H_i(Z_1) \otimes H_{4n}(Z_1) \too H_i(Z)$ is
precisely $H_i(Z_1)$.

Assume that $j>1$ and set $w=w_j$. We will complete the proof using a
commutative diagram:
\begin{equation}
  \label{eq:5}
  \xymatrix{ 0\ar[r]& H_i(Z_1) \otimes H_{4n}(Z_{j-1})
    \ar[r]^-{\id\otimes (r_j)_*} \ar[d]_{*} &H_i(Z_1) \otimes
    H_{4n}(Z_j) \ar[r]^-{\id\otimes \res_j} \ar[d]_{*} &H_i(Z_1)
    \otimes H_{4n}(Z_{w}) \ar[r] \ar[d]_{*} & 0 \\
    0\ar[r]& H_i(Z_{j-1}) \ar[r]^-{(r_j)_*} &H_i(Z_j)
    \ar[r]^-{\res_j} &H_i(Z_{w}) \ar[r]& 0} 
\end{equation}
and the Five Lemma. We saw in (\ref{eq:ses}) that the bottom row is
exact and it follows that the top row is also exact. By induction, the
convolution product in $H_\bullet(Z)$ determines a surjective map $* :
H_i(Z_1) \otimes H_{4n}(Z_{j-1}) \too H_i(Z_{j-1})$. To conclude from
the Five Lemma that the middle vertical map is a surjection, it
remains to define the other vertical maps so that the diagram commutes
and to show that the right-hand vertical map is a surjection.

First we show that the image of the map $H_i(Z_1) \otimes
H_{4n}(Z_{j}) \too H_i(Z_{j})$ determined by the convolution product
in $H_\bullet(Z)$ is contained in $H_i(Z_j)$. It then follows that the
middle vertical map in (\ref{eq:5}) is defined and so by exactness
there is an induced map from $H_i(Z_1) \otimes H_{4n}(Z_w)$ to
$H_i(Z_w)$ so that the diagram (\ref{eq:5}) commutes. Second we show
that the right-hand vertical map is a surjection.

By Lemma \ref{lemgraph}, $\Lambda_{1}^{\ell^*}$ is the graph of the
identity map of $\fgt^{\ell^*}$, and $\Lambda_w^{ \ell^*}$ is the
graph of $\wtilde|_{\fgt^{\ell^*}}$.  Therefore, $\Lambda_{1}^{\ell^*}
\circ \Lambda_w^{ \ell^*}= \Lambda_w^{ \ell^*}$ and there is a
convolution product
\[
H_{i+2}( \Lambda_{1}^{ \ell^*}) \otimes H_{4n+2}( \Lambda_w^{ \ell^*})
\xrightarrow{\,\ *\ \,} H_{i+2}( \Lambda_w^{\ell^*}).
\]
Suppose $a$ is in $H_i(Z_1)$. Then by Corollary \ref{lim1}, $a=
\lim_1(a_1)$ for some $a_1$ in $H_{i+2}( \Lambda_1^{ \ell^*})$.  It is
shown in \cite[Proposition 2.7.23]{chrissginzburg:representation} that
specialization commutes with convolution, so $\lim(a_1*
[\Lambda_w^{\ell^*}])= \lim(a_1)* \lim([\Lambda_w^{\ell^*}])=
a*\lambda_w$. Also, $a_1* [\Lambda_w^{ \ell^*}]$ is in
$H_{i+2}(\Lambda_w^{\ell^*})$ and $\lim= r_*\circ \lim_j$ and so
$a*\lambda_w = r_* \circ \lim_j( a_1* [\Lambda_w^{\ell^*}])$ is in
$H_i(Z_j)$.  By induction, if $k<j$, then $a*\lambda_{w_k}$ is in
$H_i(Z_k)$ and so $a*\lambda_{w_k}$ is in $H_i(Z_k)$. Since the set
$\{\, \lambda_{w_k}\mid 1\leq k\leq j\,\}$ is a basis of
$H_{4n}(Z_j)$, it follows that $a*H_{4n}(Z_j)\subseteq H_i(Z_j)$.
Therefore, the image of the convolution map $H_i(Z_1) \otimes
H_{4n}(Z_{j}) \too H_i(Z)$ is contained in $H_i(Z_{j})$.

To complete the proof of Lemma \ref{surjj}, we need to show that the
induced map from $H_i(Z_1) \otimes H_{4n}(Z_w)$ to $H_i(Z_w)$ is
surjective.

Consider the following diagram:
\[
\xymatrix@1{H_{i+2}( \Lambda_{1}^{ \ell^*}) \otimes H_{4n+2}(
  \Lambda_w^{ \ell^*}) \ar[r]^-{*} \ar[d]_{\lim_1 \otimes \lim_j} &
  H_{i+2}( \Lambda_w^{ \ell^*}) \ar[d]^{\lim_j} \\
  H_i(Z_1) \otimes H_{4n}(Z_j) \ar[r]^-{*} \ar[d]_{\id\otimes
    \res_j} & H_i(Z_j) \ar[d]^{\res_j} \\
  H_i(Z_1) \otimes H_{4n}(Z_{w}) \ar[r]^-{*} & H_i(Z_{w})}
\]
We have seen that the bottom square is commutative. It follows from
the fact that specialization commutes with convolution that the top
square is also commutative. It is shown in Proposition \ref{propg1}
that the convolution product $H_{i+2}( \Lambda_{1}^{ \ell^*}) \otimes
H_{4n+2}( \Lambda_w^{ \ell^*}) \to H_{i+2}( \Lambda_w^{ \ell^*})$ is
an injection. Since $H_{i+2}( \Lambda_{1}^{ \ell^*})$ is
finite-dimensional and $H_{4n+2}( \Lambda_w^{ \ell^*})$ is
one-dimensional, it follows that this convolution mapping is an
isomorphism. Also, we saw in Lemma \ref{lem3.6} that $\res_j\circ
\lim_j$ is surjective. Therefore, the composition $\res_j\circ
\lim_j\circ \,*$ is surjective and it follows that the bottom
convolution map $H_i(Z_1) \otimes H_{4n}(Z_{w}) \to H_i(Z_{w})$ is
also surjective.  This completes the proof of Lemma \ref{surjj}.

\section{Smash Product Structure}
In this section we prove Theorem \ref{conj}. We need to show that
$\lambda_w*H_i(Z_1) *\lambda_{w\inverse} = H_i(Z_1)$ and that
$\beta\colon \Coinv_\bullet(W) \xrightarrow{\ \cong\ } H_{4n-\bullet}
(Z_1)$ is an isomorphism of $W$-algebras.

Suppose that $\ell$ is a one-dimensional subspace of $\ft$ so that
$\ell^*= \ell\setminus \{0\} = \ell\cap \ft_\reg$.  Recall that for $S
\subseteq \ft$, $\fgt^{S}= \nu\inverse(S)$. By Lemma \ref{lemgraph},
if $w$ is in $W$, then $\Lambda_w^{\ell^*}$ is the graph of the
restriction of $\wtilde$ to $\fgt^{\ell^*}$.  It follows that there is
a convolution product
\[
H_{4n+2}(\Lambda_w^{\ell^*}) \otimes H_{i+2}( \Lambda_{1}^{
  w\inverse(\ell^*)}) \otimes H_{4n+2}( \Lambda_{w}^{w
  \inverse(\ell^*)}) \xrightarrow{\,\ *\ \,} H_{i+2}(
\Lambda_{1}^{\ell^*}).
\]
Because specialization commutes with convolution, the diagram
\[
\xymatrix{H_{4n+2}(\Lambda_w^{\ell^*}) \otimes H_{i+2}( \Lambda_{1}^{
    w\inverse(\ell^*)}) \otimes H_{4n+2}( \Lambda_{w}^{w
    \inverse(\ell^*)}) \ar[r]^-{*}\ar[d]_{\lim\otimes \lim \otimes
    \lim} &H_{i+2}( \Lambda_{1}^{\ell^*}) \ar[d]^{\lim}  \\
  H_{4n}(Z) \otimes H_{i}( Z_1) \otimes H_{4n}( Z) \ar[r]_-{*}
  &H_{i}(Z)}
\]
commutes. 

We saw in Corollary \ref{lim1} that $\lim_1\colon H_{i+2}(\Lambda
_1^{\ell^*}) \to H_i(Z_1)$ is surjective. Thus, if $c$ is in
$H_i(Z_1)$, then $c= \lim(c_1)$ for some $c_1$ in $H_{i+2}(
\Lambda_{w_1} ^{w\inverse(\ell^*)})$. Therefore,
\[
\lambda_w*c*\lambda_{w\inverse}= \lim ([\Lambda_w^{\ell^*}]) *\lim
(c_1)* \lim( [\Lambda_{w\inverse}^{w \inverse(\ell^*)}]) =\lim\left(
  [\Lambda_w^{\ell^*}]*c_1* [\Lambda_{w\inverse}^{w \inverse(\ell^*)}]
\right).
\]
Since $\Lambda_w^{\ell^*}$ and $\Lambda_{w \inverse}^{w
  \inverse(\ell^*)}$ are the graphs of $\wtilde$ and $\wtilde\inverse$
respectively, and $\Lambda_1^{w\inverse(\ell^*)}$ is the graph of the
identity function, it follows that $[\Lambda_w^{\ell^*}]*c_1*
[\Lambda_{w\inverse}^{w \inverse(\ell^*)}]$ is in $H_{i+2} (\Lambda_1
^{w\inverse( \ell^*)})$ and so by (\ref{eq:fact}), $\lambda_w *c*
\lambda_{ w\inverse}$ is in $H_i(Z_1)$. This shows that $\lambda_w*
H_i(Z_1) *\lambda_{w \inverse} = H_i(Z_1)$ for all $i$.

To complete the proof of Theorem \ref{conj} we need to show that if
$w$ is in $W$ and $f$ is in $\Coinv_i(W)$, then $\beta( w\cdot f)=
\lambda_w* \beta(f) *\lambda_{w\inverse}$ where $w\cdot f$ denotes the
natural action of $w$ on $f$. To do this, we need some preliminary
results.

First, since $\Lambda_1^{\ell^*}$ is the diagonal in $\fgt^{\ell^*}
\times \fgt^{\ell^*}$, it is obvious that
\[
\delta \circ \wtilde \inverse = (\wtilde \inverse \times
  \wtilde \inverse) \circ \delta\colon \fgt^{w \inverse(\ell^*)}
  \too \Lambda_1^{\ell^*}.
\]
Therefore,
\begin{equation}
  \label{eq:wcomm}
  \delta_* \circ \wtilde \inverse_* = (\wtilde \inverse
  \times \wtilde \inverse)_* \circ \delta_*\colon H_{i}(\fgt^{w
    \inverse(\ell^*)}) \too H_{i}(\Lambda_1^{\ell^*})
\end{equation}
for all $i$. (The first $\delta$ in (\ref{eq:wcomm}) is the diagonal
embedding $\fgt^{\ell^*} \cong \Lambda_1^{\ell^*}$ and the second
$\delta$ is the diagonal embedding $\fgt^{w\inverse(\ell^*)} \cong
\Lambda_1^{w\inverse(\ell^*)}$.)

Next, with $\ell\subseteq \ft$ as above, $\fgt^{\ell}= \fgt^{\ell^*}
\coprod \nu\inverse(0)= \fgt^{\ell^*} \coprod \CNt$ and the
restriction of $\nu\colon \fgt^{\ell}\to \ell$ to $\fgt^{\ell^*}$ is a
locally trivial fibration. Therefore, there is a specialization map
$\lim\nolimits_0\colon H_{i+2} (\fgt^{\ell^*}) \to H_i(\CNt)$.  Since
$\delta_*\colon H_{i+2}( \fgt^{\ell^*}) \to H_{i+2}(
\Lambda_1^{\ell^*})$ and $\delta_*\colon H_i(Z) \to H_i(Z_1)$ are
isomorphisms, the next lemma is obvious.

\begin{lemma}\label{lemlimcomm}
  Suppose that $\ell$ is a one-dimensional subspace of $\ft$ so that
  $\ell^* = \ell\setminus \{0\} \subseteq \ft_\reg$.  Then the diagram
  \[
  \xymatrix{H_{i+2}( \fgt^{\ell^*}) \ar[r]^{\delta_*}
    \ar[d]_{\lim\nolimits_0} & H_{i+2}( \Lambda_1^{\ell^*})
    \ar[d]^{\lim\nolimits_1} \\ H_i(Z) \ar[r]_{\delta_*} & H_i(Z_1)}
  \]
  commutes.
\end{lemma}

Finally, $\CNt \times_{\CN} \CN= \CNt$ and so $Z\circ \CNt= (\CNt
\times_{\CN} \CNt) \circ (\CNt \times_{\CN} \CN)= \CNt \times_{\CN}
\CN$. Thus, there is a convolution action, $H_{4n}(Z) \otimes
H_i(\CNt) \too H_i(\CNt)$, of $H_{4n}(Z)$ on $H_i(\CNt)$.

Suppose that $w$ is in $W$ and $z$ is in $H^i(\CB)$. Then $\pi^*\circ
\pd(z)$ is in $H_{4n-i}(\CNt)$ and so $\lambda_w* (\pi^*\circ \pd(z))$
is in $H_{4n-i}(\CNt)$. It is shown in \cite[Proposition
7.3.31]{chrissginzburg:representation} that for $y$ in
$H_\bullet(\CB)$, $\lambda_w*\pi^*(y)= \epsilon_w \pi^*(w\cdot y)$
where $\epsilon_w$ is the sign of $w$ and $w\cdot y$ denotes the
action of $W$ on $H_\bullet(\CB)$ coming from the action of $W$ on
$G/T$ and the homotopy equivalence $G/T \simeq \CB$. It is also shown
in \cite[Proposition 7.3.31]{chrissginzburg:representation} that $\pd(
w\cdot z)= \epsilon_w w\cdot \pd(z)$. Therefore,
\[
\lambda_w* (\pi^*\circ \pd(z))= \epsilon_w \pi^*( w\cdot \pd(z))=
\epsilon_w \epsilon_w \pi^* \circ \pd(w\cdot z)= \pi^* \circ
\pd(w\cdot z).
\]
This proves the next lemma.

\begin{lemma}\label{lemWcomm}
  If $w$ is in $W$ and $z$ is in $H_i(\CB)$, then $\lambda_w*
  (\pi^*\circ \pd(z))= \pi^* \circ \pd(w\cdot z)$.
\end{lemma}

\subsection*{Proof of Theorem \ref{conj}}
Fix $w$ in $W$ and $f$ in $\Coinv_i(W)$. Using the fact that $\beta=
\delta_* \circ \pi^* \circ \pd\circ \bi$ we compute
\begin{align*}
  \lambda_w* \beta(f) *\lambda_{w\inverse}&= \lim\nolimits_1\left(
    [\Lambda_w^{\ell^*}] *\lim\nolimits_1\inverse(\beta(f)) *[\Lambda_
    {w\inverse} ^{w\inverse(\ell^*)}] \right) &
  \text{\cite[2.7.23]{chrissginzburg:representation}} \\
  &= \lim\nolimits_1\circ (\wtilde\inverse\times \wtilde \inverse)_*
  \circ \lim\nolimits_1 \inverse \circ \beta(f) & \text{Proposition
    \ref{2to1}} \\
  &= \lim\nolimits_1 \circ \delta_* \circ \wtilde\inverse_* \circ
  \delta_*\inverse \circ \lim \nolimits_1 \inverse\circ \beta(f)
  & \text{(\ref{eq:wcomm})}\\
  &= \delta_* \circ \lim\nolimits_0 \circ \wtilde\inverse_* \circ
  \delta_*\inverse \circ \lim \nolimits_1 \inverse\circ \delta_* \circ
  \delta_*\inverse\circ \beta(f) & \text{Lemma \ref{lemlimcomm}}\\
  &= \delta_* \circ \lim\nolimits_0 \circ \wtilde\inverse_* \circ \lim
  \nolimits_0 \inverse \circ \delta_*\inverse\circ \beta(f) &
  \text{Lemma \ref{lemlimcomm}}\\
  &= \delta_* \circ \lim\nolimits_0 \circ \wtilde\inverse_* \circ \lim
  \nolimits_0 \inverse \circ \pi^* \circ \pd\circ \bi(f) & \\
  &= \delta_* \circ \lim\nolimits_0 \left( (\lim \nolimits_0 \inverse
    \circ \pi^* \circ \pd\circ \bi(f)) *[ \Lambda_{w\inverse}
    ^{w\inverse(\ell^*)}] \right)&
  \text{\cite[2.7.11] {chrissginzburg:representation}} \\
  &= \delta_* \left( (\pi^*\circ \pd\circ \bi(f)) *
    \lambda_{w\inverse}
  \right)& \text{\cite[2.7.23] {chrissginzburg:representation}} \\
  &= \delta_* \left( \lambda_w* (\pi^*\circ \pd\circ \bi(f)) \right)&
  \text{Lemma \ref{switch} and \cite[3.6.11]
    {chrissginzburg:representation}}  \\
  &= \delta_* \circ \pi^*\circ \pd(w \cdot \bi(f))& \text{Lemma
    \ref{lemWcomm}}\\
  &= \delta_* \circ \pi^*\circ \pd\circ \bi( w\cdot f)&\text{$\bi$ is
    $W$-equivariant}\\
  &= \beta(w\cdot f).
\end{align*}
This completes the proof of Theorem \ref{conj}.


\appendix
\section{Convolution and Graphs}

In this appendix we prove some general properties of convolution and
graphs.

Suppose $M_1$, $M_2$, and $M_3$ are smooth varieties, $\dim M_2=d$,
and that $Z_{1,2}\subseteq M_1\times M_2$ and $Z_{2,3} \subseteq
M_2\times M_3$ are two closed subvarieties so that the convolution
product,
\[
H_i(Z_{1,2}) \otimes H_j(Z_{2,3}) \xrightarrow{\,\ *\ \,} H_{i+j-2d}(
Z_{1,2} \circ Z_{2,3}),
\]
in \cite[\S2.7.5] {chrissginzburg:representation} is defined. For
$1\leq i,j\leq 3$, let $\tau_{i,j}\colon M_i\times M_j \to M_j\times
M_i$ be the map that switches the factors. Define $Z_{2,1}=
\tau_{1,2}(Z_{1,2}) \subseteq M_2\times M_1$ and $Z_{3,2}=
\tau_{2,3}(Z_{2,3}) \subseteq M_3\times M_2$. Then the convolution
product
\[
H_j(Z_{3,2}) \otimes H_i(Z_{2,1}) \xrightarrow{\,\ *'\ \,} H_{i+j-2d}(
Z_{3,2} \circ Z_{2,1})
\] 
is defined. We omit the easy proof of the following lemma.

\begin{lemma}\label{switch}
  If $c$ is in $H_i(Z_{1,2})$ and $d$ is in $H_j(Z_{2,3})$, then
  $(\tau_{1,3})_* (c*d)= (\tau_{2,3})_*(d) *' (\tau_{1,2})_*(c)$.
\end{lemma}

Now suppose $X$ is an irreducible, smooth, $m$-dimensional variety,
$Y$ is a smooth variety, and $f\colon X\to Y$ is a morphism. Then if
$\Gamma_X$ and $\Gamma_f$ denote the graphs of $\id_X$ and $f$
respectively, using the notation in
\cite[\S2.7]{chrissginzburg:representation}, we have $\Gamma_X \circ
\Gamma_f = \Gamma_f$ and there is a convolution product $*\colon
H_i(\Gamma_X) \otimes H_{2m} (\Gamma_f) \too H_i(\Gamma_f)$.

\begin{proposition}\label{propg1}
  The convolution product $*\colon H_i(\Gamma_X) \otimes H_{2m}
  (\Gamma_f) \too H_i(\Gamma_f)$ is an injection.
\end{proposition}

\begin{proof}
  For $i, j=1, 2, 3$, let $p_{i,j}$ denote the projection of $X\times
  X\times Y$ on the $i\th$ and $j\th$ factors. Then the restriction of
  $p_{1,3}$ to $(\Gamma_X\times Y) \cap (X\times \Gamma_f)$ is the map
  that sends $(x,x,f(x))$ to $(x, f(x))$. Thus, the restriction of
  $p_{1,3}$ to $(\Gamma_X\times Y) \cap (X\times \Gamma_f)$ is an
  isomorphism onto $\Gamma_f$ and hence is proper. Therefore, the
  convolution product in homology is defined.
  
  Since $X$ is irreducible, so is $\Gamma_f$ and so $H_{2m}(\Gamma_f)$
  is one-dimensional with basis $[\Gamma_f]$. Suppose that $c$ is in
  $H_i(\Gamma_X)$. We need to show that if $c* [\Gamma_f]=0$, then
  $c=0$.
  
  Fix $c$ in $H_i(\Gamma_X)$. Notice that the restriction of $p_{1,3}$
  to $(\Gamma_X\times Y) \cap (X\times \Gamma_f)$ is the same as the
  restriction of $p_{2,3}$ to $(\Gamma_X\times Y) \cap (X\times
  \Gamma_f)$. Thus, using the projection formula, we have
  \begin{align*}
    c*[\Gamma_f]&= (p_{1,3})_* \left( p_{1,2}^*c \cap p_{2,3}^*
      [\Gamma_f] \right) \\
    &= (p_{2,3})_* \left( p_{1,2}^*c \cap p_{2,3}^*
      [\Gamma_f] \right) \\
    &= \left((p_{2,3})_* p_{1,2}^* c\right) \cap [\Gamma_f],
  \end{align*}
  where the intersection product in the last line is from the
  cartesian square:
  \[
  \xymatrix{ \Gamma_f \ar[d] \ar[r]^{=} & \Gamma_f \ar[d] \\
    X\times Y\ar[r]_{=}& X\times Y}
  \]
  
  Let $p\colon X\times Y\to X$ and $q\colon \Gamma_X\to X$ be the
  first and second projections, respectively. Then the square
  \[
  \xymatrix{ \Gamma_X\times Y \ar[d]_{p_{1,2}} \ar[r]^{p_{2,3}} &
    X\times Y \ar[d]^{p} \\
    \Gamma_X\ar[r]_{q}& X}
  \]
  is cartesian. Thus, 
  \begin{align*}
    p_*\left( c*[\Gamma_f]\right) &= p_* \left(\left((p_{2,3})_*
        p_{1,2}^* c\right) \cap [\Gamma_f] \right) \\
    &=p_* \left( \left(p^*q_* c\right) \cap [\Gamma_f] \right) \\
    &=q_* c \cap (p|_{\Gamma_f})_*[\Gamma_f] \\
    &= q_* c \cap [X] \\
    &= q_* c,
  \end{align*}
  where we have used the projection formula and the fact that
  $(p|_{\Gamma_f})_*[\Gamma_f]= [X]$.
  
  Now if $c*[\Gamma_f]=0$, then $q_*c=0$ and so $c=0$, because $q$ is
  an isomorphism.
\end{proof}

Let $\Gamma_Y$ denote the graph of the identity functions $\id_Y$.
Then the following compositions and convolution products in
Borel-Moore homology are defined:
\begin{itemize}
\item $\Gamma_f \circ \Gamma_X= \Gamma_f$ and so there is a
  convolution product $H_i(\Gamma_f) \otimes H_j(\Gamma_X) \too
  H_{i+j-m}(\Gamma_f)$.
\item $\Gamma_Y \circ \Gamma_{f\inverse}= \Gamma_{f\inverse}$ and so
  there is a convolution product $H_i(\Gamma_X) \otimes
  H_j(\Gamma_{f\inverse}) \too H_{i+j-m}(\Gamma_{f\inverse})$.
\item $\Gamma_f \circ \Gamma_{f\inverse}= \Gamma_X$ and so there is a
  convolution product $H_i(\Gamma_f) \otimes H_j(\Gamma_{f\inverse})
  \too H_{i+j-m}(\Gamma_X)$.
\end{itemize}
Thus, if $c$ is in $H_i(\Gamma_Y)$, then $[\Gamma_f]*c* [\Gamma_{
  f\inverse}]$ is in $H_i(\Gamma_X)$. Notice that $f\inverse \times
f\inverse \colon \Gamma_Y \to \Gamma_X$ is an isomorphism, so in
particular it is proper.

\begin{proposition}\label{2to1}
  If $c$ is in $H_i(\Gamma_Y)$, then $[\Gamma_f]*c* [\Gamma_{
    f\inverse}]= (f\inverse\times f\inverse)_*(c)$.
\end{proposition}

\begin{proof}
  We compute $([\Gamma_f]*c) *[\Gamma_{f\inverse}]$, starting with
  $[\Gamma_f]*c$.

  For $1\leq i, j\leq 3$ let $q_{i,j}$ be the projection of the subset
  \[
  \Gamma_f\times Y \cap X\times \Gamma_Y =\{\, (x, f(x), f(x)) \mid
  x\in X\,\}
  \] 
  of $X\times Y\times Y$ onto the $i,j$-factors.  Then $q_{1,3}=
  q_{1,2}$.  Therefore, using the projection formula, we see that
  \begin{align*}
    [\Gamma_f]*c &=(q_{1,3})_* \left( q_{1,2}^* [\Gamma_f] \cap
      q_{2,3}^* c\right) \\
    &=(q_{1,2})_* \left( q_{1,2}^* [\Gamma_f] \cap
      q_{2,3}^* c\right) \\
    &=[\Gamma_f] \cap (q_{1,2})_* q_{2,3}^* c \\
    &= (q_{1,2})_* q_{2,3}^* c .
  \end{align*}

  Next, for $1\leq i, j\leq 3$ let $p_{i,j}$ be the projection of the
  subset
  \[
  \Gamma_f\times X \cap X\times \Gamma_{f\inverse}= \{\, (x, f(x), x)
  \mid x\in X\,\}
  \]
  of $X\times Y\times X$ onto the $i,j$-factors. Then $p_{1,3}=
  (f\inverse \times \id)\circ p_{2,3}$. Therefore, using the fact that
  $[\Gamma_f]*c = (q_{1,2})_* q_{2,3}^* c$ and the projection formula,
  we have
  \begin{align*}
    ([\Gamma_f]*c)* [\Gamma_{f\inverse}] &=(p_{1,3})_* \left(
      p_{1,2}^*((q_{1,2})_* q_{2,3}^* c ) \cap p_{2,3}^*
      [\Gamma_{f\inverse}] \right) \\
    &=(f\inverse \times \id)_* (p_{2,3})_* \left(
      p_{1,2}^*((q_{1,2})_* q_{2,3}^* c ) \cap p_{2,3}^* [\Gamma_{
        f\inverse}] \right) \\
    &=(f\inverse \times \id)_* \left( (p_{2,3})_* p_{1,2}^*
      (q_{1,2})_* q_{2,3}^* c \cap [\Gamma_{f\inverse}] \right) \\
    &=(f\inverse \times \id)_* (p_{2,3})_* p_{1,2}^* (q_{1,2})_*
    q_{2,3}^* c .
  \end{align*}

  The commutative square
  \[
  \xymatrix{ \Gamma_f\times X \cap X\times \Gamma_{f\inverse}
    \ar[rr]^{\id\times \id\times f} \ar[d]_{\id} && \Gamma_f\times Y
    \cap X\times \Gamma_Y \ar[d]^{q_{1,2}} \\
    \Gamma_f\times X \cap X\times \Gamma_{f\inverse}
    \ar[rr]_-{p_{1,2}} && \Gamma_f}
  \]
  is cartesian, so $p_{1,2}^* (q_{1,2})_*= (\id\times \id\times f)^*$.

  Also, the commutative square
  \[
  \xymatrix{ \Gamma_f\times X \cap X\times \Gamma_{f\inverse}
    \ar[rrr]^-{q_{2,3} \circ (\id\times \id\times f)}
    \ar[d]_{(f\inverse \times \id) \circ p_{2,3}} &&& \Gamma_Y
    \ar[d]^{f\inverse \times f\inverse} \\
    \Gamma_X \ar[rrr]_{\id} &&& \Gamma_X}
  \]
  is cartesian, so $(f\inverse \times \id)_* (p_{2,3})_* (\id\times
  \id\times f)^* q_{2,3}^*= (f\inverse \times f\inverse)_*$.

  Therefore,
  \begin{align*}
    ([\Gamma_f]*c)* [\Gamma_{f\inverse}] &= (f\inverse \times \id)_*
    (p_{2,3})_* p_{1,2}^* (q_{1,2})_* q_{2,3}^* c\\
    &= (f\inverse \times \id)_* (p_{2,3})_* (\id\times \id\times
    f)^* q_{2,3}^* c\\
    &= (f\inverse \times f\inverse)_* c.
  \end{align*}
  This completes the proof of the proposition.
\end{proof}


\bigskip

\bibliographystyle{amsplain}

\providecommand{\bysame}{\leavevmode\hbox to3em{\hrulefill}\thinspace}
\providecommand{\MR}{\relax\ifhmode\unskip\space\fi MR }
\providecommand{\MRhref}[2]{%
  \href{http://www.ams.org/mathscinet-getitem?mr=#1}{#2}
}
\providecommand{\href}[2]{#2}


\end{document}